\documentclass[12pt,reqno]{amsart}

\usepackage[arrow,matrix,curve]{xy}

\usepackage[dvips]{graphicx} 

\usepackage{amssymb, latexsym, amsmath, amscd, array, hyperref
}

\usepackage{breakurl}   

\newtheorem{theorem}{Theorem}[section]

\newtheorem{corollary}[theorem]{Corollary}

\theoremstyle{definition}

\newtheorem{question}[theorem]{Question}

\numberwithin{equation}{section}

\newcommand\R {{\mathbb R}}

\newcommand\Z {{\mathbb Z}} 
\newcommand\CP {{\mathbb C\mathbb P}}

\DeclareMathOperator\sys{{\rm sys}}
\DeclareMathOperator\vol{{\rm vol}}
\DeclareMathOperator\area{{\rm area}}
\DeclareMathOperator\stsys{{\rm stsys}}
\DeclareMathOperator\LCG{{\rm LCG}}
\DeclareMathOperator\length{{\rm length}}

\author{Mikhail G. Katz} \address{M. Katz, Department of Mathematics,
Bar Ilan University, Ramat Gan 5290002 Israel}
\email{katzmik@macs.biu.ac.il}

\begin{document}

\thispagestyle{empty}


\title [An inequality for complex projective plane]
{An inequality for length and volume in the complex projective plane}

\begin{abstract}
We prove a new inequality relating volume to length of closed
geodesics on area minimizers for generic metrics on the complex
projective plane.
We exploit recent regularity results for area minimizers by Moore and
White, and the Kronheimer--Mrowka proof of the Thom conjecture.
\end{abstract}



\maketitle

\section{Introduction}
\label{s1}

The~$1$-systole~$\sys_1$ of a Riemannian manifold~$M$ is the least
length of a noncontractible loop in~$M$.  In the 1950s, Carl Loewner
proved an inequality relating the systole and the area of an arbitrary
metric on the~$2$-dimensional torus; see \cite{Pu52}, \cite{Gr99},
\cite{Ka07a}.

Gromov obtained a variety of inequalities relating the~$1$-systole and
the volume of~$M$.  We will now present those that will be used in
this paper.  Thus, we have the inequality
\begin{equation}
\label{e17}
\sys_1^2(S)\leq \frac43\area(S)
\end{equation}
from \cite[p.\;49, Corollary 5.2.B]{Gr83}, valid for every closed
aspherical surface~$S$.  There are also bounds that improve as the
genus~$g$ grows.  Thus, Gromov proved that a surface~$S_g$ of
genus~$g$ satisfies
\begin{equation}
\label{s12}
\sys_1^2(S_g) \leq \frac{64}{4\sqrt{g}+27} \area(S_g)
\end{equation}
(cf.\;\cite[p.\;1216]{Ka05}, \cite{Ko87}).  An asymptotically better
result is the following:
\begin{equation} 
\label{e13d}
\frac{\sys_1^2(S_g)}{\area} \leq \frac1\pi\, \frac {\log^2
g}{g}(1+o(1)) \mbox{ when } g\to \infty
\end{equation}
from \cite[p.\;1211, Theorem 2.2]{Ka05}, improving the multiplicative
constant in Gromov's similar upper bound.  Apart from the
multiplicative constant, the asymptotic behavior~$\frac{\log^2 g}{g}$
is the correct one due to the existence of arithmetic hyperbolic
surfaces satisfying \emph{lower} bounds of this type; see e.g.,
Buser--Sarnak~\cite{Bu94}, Katz et al.\ \cite{Ka07}, \cite{Ka11},
\cite{Ka16}.

The literature contains a number of results on the higher systoles, as
well.  Recall that the~$2$-systole~$\sys_2$ of~$M$ can be defined as
the least area of a homologically nontrivial surface in~$M$, or more
generally~$2$-cycle with integer coefficients:
\[
\sys_2(M)=\min \left\{ \area(S) \colon [S]\in H_2^{\phantom{I}}(M;\Z)
\setminus \{0\} \right\}.
\]
The~$2$-systole of~$M$ is typically not controlled by the volume
of~$M$, a phenomenon referred to as \emph{systolic freedom}.  For
example, the complex projective plane~$\CP^2$ admits metrics of
arbitrarily small volume such that every homologically nontrivial
surface in~$\CP^2$ has at least unit area; see Katz--Suciu
(\cite{Ka99}, 1999, Theorem\;1.1, p.\;113).

Gromov has also defined a modified invariant of~$M$ called the stable
$2$-systole,~$\stsys_2(M)$.  It is a result of Federer \cite{Fe69}
that the limits below exist and therefore can be used to
define~$\stsys_2$ as follows:
\begin{equation}
\label{e13c}
\stsys_2(M)=\min\left\{\lim_{n\to\infty}\tfrac{1}{n}\sys_2(n\alpha)
\colon\alpha\in H_2(M;\Z)\setminus \{\text{torsion}\} \right\}
\end{equation}
where~$\sys_2(x)$ denotes the infimum of areas of surfaces
representing the class~$x\in H_2(M;\Z)$.  Gromov's stable systolic
inequality for~$\CP^n$ asserts that
\begin{equation}
\label{e12}
\stsys_2^n(\CP^n) \leq n! \, \vol(\CP^n),
\end{equation}
for an arbitrary metric on~$\CP^n$.  A more general non-optimal
inequality appears in \cite[p.\;96, item 7.4.C]{Gr83}.  The optimal
inequality for $\CP^n$ is in \cite[p.\;262, Theorem\;4.36]{Gr99}.

The proof relies on the duality of the stable norm in homology and the
comass norm in cohomology; see \cite[Section 4.34, p.\;261]{Gr99},
\cite[Section\;4.10, p.\;380]{Fe74}, \cite[Lemma\;17]{Pa99},
\cite{Ba07}, \cite{Ka06}, \cite{Iv04}.  The 2-point homogeneous
Fubini--Study metric satisfies the case of equality in this sharp
inequality.

The~$1$-systole of a compact manifold~$M$ is always realized by a
closed geodesic in~$M$.  There exist inequalities for the least
length~$\LCG(M)$ of a nontrivial closed geodesic in~$M$.  One such
inequality is the Croke--Rotman inequality
\begin{equation}
\label{s15}
\LCG^2(S^2)\leq32\area(S^2)
\end{equation}
(see \cite{Cr88}, \cite{Ro06}), for an arbitrary metric on the
$2$-sphere.  The multiplicative constant~$32$ is not believed to be
optimal.  It is conjectured that the following tight bound holds:
\begin{equation}
\label{e16b}
\LCG^2(S^2)\leq 2\sqrt{3} \area(S^2)
\end{equation}
(see e.g., Burns--Matveev \cite[p.\;23]{Bu19}), with equality attained
by a triangular ``pillow.''

We will exploit these geometric inequalities to prove
Theorem~\ref{t11b} below.

Note that a degree~$d$ Veronese embedding~$\CP^1\to\CP^{n(d)}$ (by
degree~$d$ homogeneous polynomials in homogeneous coordinates) has
constant Gaussian curvature.  Namely, the embedding preserves the
Kahler form~$\omega$.  Hence it necessarily preserves the metric due
to the relation $\omega(X,Y)=g(JX,Y)$ (when the metrics are suitably
normalized).  Hence the Veronese embedding is an isometry.  Its image
has area~$\pi d$ when the sectional curvature of~$\CP^n$ is normalized
to satisfy \mbox{$1\leq K\leq 4$}.  The geodesics are the great
circles and therefore the~$\LCG$ of this degree~$d$ genus $0$ minimal
surface grows as~$\sqrt{d}$.  We show that such a phenomenon cannot
occur in~$\CP^2$ even if one allows arbitrary generic metrics.

\begin{theorem}
\label{t11b}
For a generic Riemannian metric on the complex projective
plane~$\CP^2$, every minimal surface~$S$ of sufficiently high
homological multiplicity admits a nontrivial closed geodesic of length
controlled by the total volume of the metric:~$\LCG^4(S)\leq C
\vol(\CP^2)$.
\end{theorem}

Specific values for the constant $C$ are discussed below.

\begin{corollary}
For a generic metric of unit volume on~$\CP^2$, there exist
homologically nontrivial minimal surfaces~$S$ with a nontrivial closed
geodesic of length at most~$C$, with constant~$C$ independent of the
metric.
\end{corollary}

Thus, while the total volume does \emph{not} control the area of a
homologically nontrivial minimizing surface, it does control the
length of a shortest closed geodesic on the minimizer.  

More detailed versions of the theorem with explicit constants appear
in Section~\ref{s2}.  

Our techniques do not enable us to find closed geodesics of controlled
length for an arbitrary metric on~$\CP^2$.  The existence of closed
geodesics in compact Riemannian manifolds was proved by Fet
\cite{Fe52}, \cite{Fe53}, but it is unknown whether there exists a
closed geodesic of length controlled by the volume.  Sabourau
\cite{Sa04} proves the existence of a geodesic loop of length
controlled by the volume; see also Rotman \cite[Theorem\;0.3]{Ro19}
for related results.  Sabourau \cite{Sa19} constructs a one-cycle
sweepout of the essential sphere in the complex projective space
(endowed with an arbitrary Riemannian metric) whose one-cycle length
is controlled by the volume.  A generalisation appears in \cite{Na20}.

We will use the following result of Moore \cite[p.\;279,
Theorem\;5.1.1]{Mo17}, \cite[Theorem\;2]{Mo20} and
White~\cite[Corollary\;39]{Wh19}, exploiting Moore \cite{Mo06, Mo07}.

\begin{theorem}[Moore; White]
\label{t12}
For a generic metric on~$M$, a simple (multiplicity~$1$) area
minimizer~$S$ in a homology class in~$H_2(M;\Z)$ is a smoothly
embedded surface.
\end{theorem}

Briefly, by Morgan's result \cite{Mo82}, in the orientable case the
tangent planes to a minimizer~$S$ at a self-intersection~$p\in M$ must
necessarily be complex lines of a common complex structure on~$T_p M$,
and it can be shown that such a situation cannot arise for a generic
conformal factor.

The nonorientable case (for homology classes with~$\Z_2$ coefficients)
was treated in White's earlier paper \cite{Wh85}.  Our
Theorem~\ref{t11b} can be contrasted with the following open
questions in the case of~$\Z_2$ coefficients.
\begin{enumerate}
\item
Is the area of a minimizing surface~$S\in[\CP^1]\in H_2(\CP^2;\Z_2)$
controlled by the total volume: 
\[
\area(S)\leq C \sqrt{\vol(\CP^2)} \quad ?
\]
\item
Is the length of a shortest closed geodesic~$\gamma$ on a minimizing
surface~$S\in [\CP^1]\in H_2(\CP^2;\Z_2)$ controlled by the total
volume: 
\[
\length(\gamma) \leq C \sqrt[4]{\vol(\CP^2)} \quad ?
\]
\end{enumerate}
Note that a minimizing surface~$S\in[\CP^1]\in H_2(\CP^2;\Z_2)$ can be
nonorientable.  In such a situation, smoothness results of White still
apply, but there is no analog of the stable 2-systolic inequality
\eqref{e12}.

Our proof combines the regularity results by Moore and White with the
celebrated Kronheimer--Mrowka proof of the Thom conjecture (see
Section~\ref{s32}), to obtain upper bounds on the closed geodesic as
above.  Possible generalisations are discussed in Section~\ref{s32}.

\section{Closed geodesics on area minimizers}
\label{s2}

One possible tool in investigating geometric inequalities is the
existence of minimizing surfaces in homology classes.  Such existence
is guaranteed by Geometric Measure Theory; see Chang (\cite{Ch88},
1988), De Lellis et al.\ \cite{De15, De17, De18}.  It is in the nature
of the techniques used that the geometry of the minimizer is rather
inexplicit.  Geometric inequalities (such as \eqref{e11} below) can be
viewed as providing some control over the geometry of the minimizer.

While the volume of a metric on~$\CP^2$ does not control the
$2$-systole (as discussed in Section~\ref{s1}), it turns out that the
volume does control suitable 1-dimensional invariants (such as the
LCG) of area minimizing surfaces~$S\subseteq\CP^2$.

\begin{theorem}
\label{t11}
Consider a generic metric on~$\CP^2$.  Then for~$n$ sufficiently
large, an area-minimizing surface~$S\in n[\CP^1]$ satisfies
\begin{equation}
\label{e11}
\LCG(S) \leq 8\sqrt2 \sqrt[4]{2!\vol(\CP^2)}.
\end{equation}
\end{theorem}

Note that we left the factor~$2!=2$ under the radical sign in
inequality~\eqref{e11} as an allusion to Gromov's stable systolic
inequality~\eqref{e12}, used in the proof of \eqref{e11}.

\begin{corollary}
For a generic metric of unit volume on~$\CP^2$, there exist
homologically nontrivial minimal surfaces with a closed geodesic of
length at most~$8\cdot2^{3/4}$.
\end{corollary}

\begin{proof}[Proof of Theorem~\ref{t11}]
We first consider the two values~$n=1,2$.  Suppose we have a bound
\begin{equation}
\label{e13b}
\sys_2(2[\CP^1])\leq 4\stsys_2(\CP^2).
\end{equation}
Then the area of a minimizing surface~$S\in n[\CP^1]$,~$n=1,2$, is
controlled by the volume of~$\CP^2$ via Gromov's stable systolic
inequality~\eqref{e12}.  If~$S$ is the sphere, then~$\LCG(S)$ is
controlled by the area of~$S$ via the Croke--Rotman inequality
\eqref{s15}.  If~$S$ is aspherical, then even better bounds exist such
as Gromov's inequality \eqref{e17}.  By \eqref{e13b} and~\eqref{s15},
in either case we have
\[
\LCG^2(S)\leq2^7\stsys_2(\CP^2).
\]
It follows that
\[
\LCG^{4}(S)\leq2^{14}\stsys_2^2(\CP^2)
\leq2^{14}2!\vol(\CP^2),
\]
proving inequality~\eqref{e11} in these cases.

It remains to consider the case when~\eqref{e13b} is violated; in
other words,
$\sys_2(2[\CP^1])>4\stsys_2(\CP^2)$.  The rest of the proof is as in
the proof of the sharper bound of Theorem~\ref{t22} below.
\end{proof}

\section
{Inequality in the assumption of the tight Croke--Rotman inequality}
\label{s32}

\begin{theorem}
\label{t22}
Suppose the tight Croke--Rotman inequality \eqref{e16b} holds.  Then
given a generic metric on~$\CP^2$, for~$n$ sufficiently large, an
area-minimizing surface~$S\in n[\CP^1]$ satisfies
\begin{equation}
\label{e11d}
\LCG(S) \leq 2^{\frac74}_{\phantom{I}} \sqrt[4]{2!\vol(\CP^2)}.
\end{equation}
\end{theorem}

In other words, we have 
\begin{equation}
\label{lcg}
\LCG(S)\leq4\,\sqrt[4]{\vol(\CP^2)}.
\end{equation}

\begin{corollary}
For a generic metric of unit volume on~$\CP^2$, in the assumption of
\eqref{e16b}, there exist homologically nontrivial minimal surfaces
with a closed geodesic of length at most~$4$.
\end{corollary}

Note that for the Fubini--Study metric on~$\CP^2$, one
has
\[
\frac{\LCG}{\sqrt[4]{\vol}}=\frac{\pi}{\sqrt[4]{\pi^2/2}}
=\sqrt{\pi\sqrt{2}},
\]
not far from the value of the constant in \eqref{lcg}.

\begin{proof}[Proof of Theorem~\ref{t22}]
We first consider the two values~$n=1,2$.  Let~$\epsilon>0$ (to be
specified later).  Suppose we have the bound
\begin{equation}
\label{e24}
\sys_2(2[\CP^1])\leq 2(1+\epsilon)\stsys_2(\CP^2).
\end{equation}
Then the area of a minimizing surface~$S\in n[\CP^1]$,~$n=1,2$, is
controlled by the volume of~$\CP^2$ via Gromov's
inequality~\eqref{e12}.  If~$S$ is the sphere, then~$\LCG(S)$ is
controlled by the area of~$S$ via the tight Croke--Rotman inequality
\eqref{e16b} by the hypothesis of our theorem.  If~$S$ is aspherical,
then even better bounds exist such as Gromov's inequality~\eqref{e17}.
In either case, by \eqref{e16b} and \eqref{e24} we have
\[
\LCG^2(S)\leq 4\sqrt{3}(1+\epsilon) \stsys_2(\CP^2).
\]
Gromov's stable systolic inequality~\eqref{e12} then gives
\[
\begin{aligned}
\LCG^{4}(S) &\leq \left(4\sqrt3(1+\epsilon)\right)^2
\stsys_2^2(\CP^2) \\&\leq \left(4\sqrt3(1+\epsilon)\right)^2
2!  \vol(\CP^2).
\end{aligned}
\]
Thus,
\begin{equation}
\label{e25}
\LCG(S)\leq2\sqrt[4]{3}\sqrt{1+\epsilon}\sqrt[4]{2!\vol(\CP^2)}.
\end{equation}
For~$\epsilon$ sufficiently small, \eqref{e25} implies
inequality~\eqref{e11d} in these cases.

Thus, we can assume that inequality \eqref{e24} is violated; in other
words,
\[
\sys_2(2[\CP^1])>2(1+\epsilon)\stsys_2(\CP^2).
\]
In this case we will exploit higher multiples of the class~$[\CP^1]$
in order to prove the bound \eqref{e11d}.  Let~$n\geq3$ be the least
value such that
\begin{equation}
\label{e13}
\sys_2(n[\CP^1]) \leq (1+\epsilon)n\stsys_2(\CP^2).
\end{equation}
Such an~$n$ exists due to the existence of the limit in~\eqref{e13c}.
Then an area minimizer~$S\in n[\CP^1]$ is simple (of multiplicity~$1$)
and connected.  By White's Theorem~\ref{t12} for generic metrics,~$S$
is smoothly embedded.  For such surfaces, the genus~$g$ satisfies
\begin{equation}
\label{e16}
g\geq\frac12(n-1)(n-2)
\end{equation}
by Kronheimer--Mrowka \cite{Kr94}.  Since~$n\geq3$, the
bound~\eqref{e16} implies that the surface~$S$ is aspherical, and
that~$g\geq \frac{(n-2)^2}{2}$ and therefore
\begin{equation}
\label{e14}
\sqrt g\geq \frac{n-2}{\sqrt{2}}.
\end{equation}
Now bounds~\eqref{s12} and \eqref{e14} imply that
\begin{equation}
\label{e19}
\begin{aligned}
\sys_1(S)^2 &\leq \frac{64}{\frac{4}{\sqrt{2}}(n-2)+27}\area(S) \\&=
\frac{64}{\frac{4n}{\sqrt{2}} -4\sqrt2+27}\area(S)
\\&<\frac{8\sqrt{2}}{n}\area(S).
\end{aligned}
\end{equation}
Therefore by \eqref{e13} and \eqref{e19}, we have
\[
\sys_1^2(S)\leq\frac{8\sqrt{2}\area(S)}{n}
\leq8\sqrt{2}(1+\epsilon)\stsys_2(\CP^2).
\]
Thus
\[
\begin{aligned}
\sys_1^{4}(S) &\leq \big(8\sqrt2(1+\epsilon)\big)^2 \stsys_2^2(\CP^2)
\\&\leq \big(8\sqrt2(1+\epsilon)\big)^2 \,2!\vol(\CP^2)
\end{aligned}
\]
by Gromov's stable systolic inequality.  Thus
\[
\sys_1(S)\leq2^{\frac74}_{\phantom{I}} (1+\epsilon)^{\frac12}
\big(2!\vol(\CP^2)\big)^{\frac14}
\]
for generic metrics on~$\CP^2$ in this case.  As~$\epsilon\to0$ we
obtain the bound of Theorem~\ref{t22} with a constant arbitrarily
close to the required one.  In principle~$n$ (and therefore also a
connected simple minimizer~$S\in n[\CP^1]$) could depend
on~$\epsilon$.  However, for sufficiently large~$n$ we can exploit the
bound~\eqref{e13d} with better asymptotic behavior than Gromov's
bound~\eqref{s12}, proving the inequality.
\end{proof}

\begin{question}
Does the inequality admit a generalisation for~$2$-dimen\-sional
classes with positive self-intersection for arbitrary Riemannian
metrics on Kahler manifolds exploiting \cite{Mo96}, or even on
symplectic manifolds using \cite{Oz00} ?
\end{question}

\section*{Acknowledgments}

The author is grateful to Alex Nabutovsky for suggesting the question
of possible inequalities relating 1-dimensional invariants and the
volume of~$\CP^2$.

\end{document}